\documentstyle[amssymb,amsfonts,12pt]{amsart}
\newtheorem{theorem}{Theorem}[section]
\newtheorem{lemma}[theorem]{Lemma}
\newtheorem{corollary}[theorem]{Corollary}
\newtheorem{proposition}[theorem]{Proposition}
\newtheorem{definition}[theorem]{Definition}

\theoremstyle{remark}
\newtheorem{remark}[theorem]{Remark}

\def\QSet{\mbox{\rm\kern.24em
\vrule width.03em height1.48ex depth-.051ex \kern-.26em Q}}

\def\E{{\mathbb E}}
\def\P{{\bf P}}
\def\D{{\mathcal D}}
\def\T{{\bf T}}
\def\S{{\bf S}}
\def\Z{{\bf Z}}
\def\R{{\bf R}}

\def\size{{\rm size}}
\def\diam{{\rm diam}}

\def\\P{{\mathcal P}}

\def\F{{\mathcal F}}

\def\X{{\mathcal A}}
\def\O{{\mathcal B}}
\def\I{{\mathcal I}}

\def\bas{\begin{align*}}
\def\eas{\end{align*}}
\def\bi{\begin{itemize}}
\def\ei{\end{itemize}}
\newenvironment{proof}{\noindent {\bf Proof} }{\endprf\par}
\def \endprf{\hfill  {\vrule height6pt width6pt depth0pt}\medskip}
\def\emph#1{{\it #1}}

\def\Wset{{\bf W}}

\def\<{\left<}
\def\>{\right>}

\title{The Walsh model for $M_2^*$ Carleson}

\begin{document}
\author{Ciprian Demeter}
\address{Department of Mathematics, UCLA, Los Angeles CA 90095-1555}
\email{demeter@@math.ucla.edu}

\author{Michael Lacey}
\address{ School of Mathematics,
Georgia Institute of Technology, Atlanta,  GA 30332 USA}
\email{lacey@@math.gatech.edu}

\author{Terence Tao}
\address{Department of Mathematics, UCLA, Los Angeles CA 90095-1555}
\email{tao@@math.ucla.edu}

\author{Christoph Thiele}
\address{Department of Mathematics, UCLA, Los Angeles CA 90095-1555}
\email{thiele@@math.ucla.edu}

\thanks{The first author was supported by NSF Grant DMS-0556389}
\thanks{The second author was supported by an NSF Grant}
\thanks{The third author was supported by a grant from the Macarthur Foundation}
\thanks{The fourth author was supported by NSF Grant DMS-0400879}


\begin{abstract}
We study the Walsh model of a certain maximal truncation of 
Carleson's operator related to the Return Times Theorem.
\end{abstract}

\maketitle

\section{Introduction}
Let $\D$ denote the collection of all the dyadic intervals of the form $[2^im,2^i(m+1))$, $i,m\in\Z$ and let $C^{\D}(\R_+)$ be the set of all the functions $f:\R_+\to\R$ that are finite linear combinations of characteristic functions of dyadic intervals. 

For $l\ge 0$ we recall that the $l-$th Walsh function $W_l$ is defined recursively by the formula
$$W_0=1_{[0,1)}$$
$$W_{2l}=W_l(2x)+W_l(2x-1)$$
$$W_{2l+1}=W_l(2x)-W_l(2x-1).$$
We recognize that $W_1$ is the Haar function also denoted by $h$.

\begin{definition}
A tile $P$ is a rectangle $I_P\times \omega_P $ of area one, such that $I_P$ and $\omega_P$ are dyadic intervals. If $P=[2^in,2^i(n+1))\times [2^{-i}l,2^{-i}(l+1))$ is such a tile, we define the corresponding Walsh wave packet $w_P$ by
$$w_P(x)=2^{-i/2}W_l(2^{-i}x-n).$$
\end{definition}
The  intervals $I_P$ and $w_P$ will be referred to as the time and frequency intervals of the tile $P$.
\begin{definition}
A bitile $P$ is a rectangle $I_P\times \omega_P $ of area two, such that $I_P$ and $\omega_P$ are dyadic intervals. 
For any bitile 
$$P=[2^in,2^i(n+1))\times [2^{-i+1}l,2^{-i+1}(l+1))$$
we define the lower tile
$$P_1=[2^in,2^i(n+1))\times [2l2^{-i},(2l+1)2^{-i})$$
and the upper tile
$$P_2=[2^in,2^i(n+1))\times [(2l+1)2^{-i},(2l+2)2^{-i}).$$
If $\omega_P$ is the frequency interval of the bitile $P$ then we will use the notations $\omega_{P,1}$ and $\omega_{P,2}$ for the the frequency intervals of the sub-tiles $P_1$ and $P_2$.
\end{definition}

We next recall the definition of the Walsh-Fourier transform. Except on a set of measure 0 (which we shall always ignore), every $x\in \R_{+}$ can be identified with a doubly-infinite set of binary digits
$$x=...a_2a_1a_0.a_{-1}a_{-2}...$$
where $a_n\in\Z_2$ and $a_n$ is eventually zero as $n\to\infty$. We define two operations on $\R_{+}$ by
$$a_n(x\oplus y):=a_n(x)+a_n(y)$$
$$a_n(x\otimes y):=\sum_{m\in\Z}a_m(x)a_{n-m}(y),$$
where the addition and multiplication in the right hand terms are considered modulo 2. We next define the function $e:\R_{+}\to\{-1,1\}$ to be 1 when $a_{-1}=0$ and $-1$ when $a_{-1}=1$. Using this we can introduce the Walsh-Fourier transform of a function $f\in C^{\D}(\R_+)$ to be
$$\widehat{f}(\xi):=\int e(x\otimes\xi)f(x)dx.$$ We also note that the inverse Walsh-Fourier transform ${f}\check{\ }$ and the Walsh-Fourier transform coincide in this context.

In the following we will denote with $\S^{\operatorname{univ}}$ the collection of all the bitiles. 
It is known, see \cite{Th1}, that the almost everywhere convergence of the Walsh series for $f\in L^p$
$$\sum_{l\ge 0}\<f,W_l\>W_l(x)$$
is a consequence of the estimate
\begin{equation*}
\|\Wset f\|_p\lesssim \|f\|_{p},
\end{equation*}
where
$$\Wset f(x)= \|\sum_{P\in\S^{\operatorname{univ}}} \<f,w_{P_1}\> w_{P_1}(x) 1_{\omega_{P,2}}(\theta)\|_{L^{\infty}_{\theta}}.$$

Define the $M_2^*$ norm of a family of Walsh multipliers $m_k$ as
$$\|(m_k(\theta))_{k\in\Z}\|_{M_2^*(\theta)}= \sup_{\|g\|_2=1} 
\|\sup_k |(\widehat{g} m_k){\check{\ }}(x)|\|_{L_x^2}.$$
In this paper we will be concerned with getting estimates for the operator 
$${\bf W}^{\operatorname{max}}f(x)= \|(\sum_{P\in\S^{\operatorname{univ}}:|I_P|< 2^k} \<f,w_{P_1}\> w_{P_1}(x) 1_{\omega_{P,2}}(\theta))_{k\in\Z}\|_{M_2^*(\theta)}.$$

\begin{theorem}
\label{main}
For each $1<p<\infty$ we have
\begin{equation}
\label{e.mainy1}
\|{\bf W}^{\operatorname{max}}f\|_p\lesssim_p\|f\|_p.
\end{equation}
\end{theorem}

It has been acknowledged, see for example \cite{MTT6}, \cite{Th2}, that the Walsh models provide a lot of the intuition that lies behind their Fourier analog. In our case, the interest in proving Theorem \ref{main} is motivated by its connections with the following  Return Times Theorem due to Bourgain \cite{Bo5}.  

\begin{theorem}
\label{Bretthm}
Let ${\bf X}=(X,\Sigma,\mu, \tau)$ be a dynamical system and let $1\le p,q\le \infty$ satisfy $\frac{1}p+\frac1q\le 1$. For each function $f\in L^{p}(X)$ there is a universal set $X_0\subseteq X$ with $\mu(X_0)=1$, such that for each second dynamical system  ${\bf Y}=(Y,\F,\nu,\sigma)$, each $g\in L^{q}(Y)$ and each $x\in X_0$, the averages
$$\frac1{N}\sum_{n=0}^{N-1}f(\tau^nx)g(\sigma^ny)$$
converge $\nu$- almost everywhere.
\end{theorem}

In \cite{DLTT} we extend Bourgain's theorem to a larger range of $p$ and $q$. Our argument there relies on estimates like the one in Theorem \ref{Bretthm} for a model operator  which is the Fourier counterpart of ${\bf W}^{\operatorname{max}}$. We hope that our presentation here for the simpler Walsh model will ease the understanding of the the proof in \cite{DLTT}. 

We note that in order to prove Theorem ~\ref{main} it suffices to assume that the summation in the definition of the operator ${\bf W}^{\operatorname{max}}$ runs over a finite collection $\S\subset\S^{\operatorname{univ}}$ of bitiles, and to prove inequality ~\eqref{e.mainy1} with bounds independent on $\S$. We fix the collection $\S$ for the remaining part of the paper.

The argument relies on first splitting the collection of bitiles into structured collections called trees. The bitiles in each tree give rise to a modulated Littlewood-Paley decomposition. The model operator ${\bf W}^{\operatorname{max}}$ restricted to each  such a tree is estimated in Section \ref{generalfacts}, by using the Cald\'eron-Zygmund-type estimates from  Section \ref{var}.

In Section ~\ref{treeestimates} the operator ${\bf W}^{\operatorname{max}}f(x)$ is estimated pointwise, and it is shown that for each $x$ the contribution to ${\bf W}^{\operatorname{max}}f(x)$ comes from one stack of trees. Crucial to estimating this contribution is a weighted version of a maximal multiplier result due to Bourgain. This is proved in Section ~\ref{Bourgain}. The different pieces of the proof are put together in the last section of this paper.

\section{Variational norm estimates for averages}
\label{var}
Let $H$ be a separable Hilbert space equipped with a norm $|\cdot|_H$ and denote by $L^q(\R,H)$
the measurable functions on $\R$ with values in $H$ whose $q$-th power are integrable.
Let $\E(f|\D_k)$  denote the  conditional expectation with respect to the $\sigma$-algebra on $\R$ generated by the dyadic intervals of length $2^{k}$.
We include the case $k=\infty$ by setting  $\E(f|\D_\infty)=0$. From now on we will use the notation $$g_I(x)=\frac{1}{|I|^{1/2}}g(\frac{x-l(I)}{|I|})$$
for each  dyadic interval $I=[l(I),r(I))$.

\begin{lemma}[Jump inequality]
\label{jumplemma}
Consider $1<q<\infty$ and $f\in L^q(\R,H)$.
For each $x$ and $\lambda>0$ define the entropy number $M_\lambda(x)$ be the maximal length of a chain $\infty=k_0>k_1>k_2>\dots>k_{M_\lambda(x)}$ such that
for each $1\le m\le M_\lambda(x)$
$$|\E(f|\D_{k_m})(x)-\E(f|\D_{k_{m-1}})(x)|_H\ge \lambda.$$
Then
$$\|\lambda M_\lambda^{1/2}(x)\|_{L_x^q(\R,H)}\le C_q\|f\|_{L^q(\R,H)}$$
where the constant $C_q$ remains bounded for $q$ in any compact subinterval
of $(1,\infty)$.
\end{lemma}

\begin{proof}
This result is well known, we briefly sketch the proof for completeness. 
First we establish that the number $M_\lambda(x)$ of $\lambda$-jumps
can be estimated by counting the $\lambda/2$-jumps in a greedy 
algorithmic way.
Let $k_0(x)=\infty$ and for $m\ge 1$ let $k_m(x)$ be the minimal number,
if it exists, such that
$|\E(f|\D_{k_m(x)})(x)-\E(f|\D_{k_{m-1}(x)})(x)|_H\ge \lambda/2$. Let $\tilde{M}_\lambda(x)$ be the maximal index for which $k_{\tilde{M}_\lambda(x)}$ exists. Define 
$$A_x=\{k_0(x),k_1(x),\ldots,k_{\tilde{M}_\lambda(x)}\}$$
$$\I_x=\{J\in\D:\,x\in J,\,|J|=2^{k}\,\hbox{for some}\,k\in A_x\}.$$
Then one easily checks that $M_\lambda(x)\le \tilde{M}_\lambda(x)$.
The crucial additional property of this greedy selection is 
that the initial parts of the sequence $k_m$ coincide for two
nearby values of $x$ until the value of $2^{k_m}$ gets smaller than 
the length of the smallest dyadic interval containing both values.

For each $x$ and each selected interval $J\in\I_x$,  let $\I_J$ be the collection of dyadic intervals contained in $J$ but not contained in any  interval from $\I_x$ of length smaller than $|J|$.
By vector valued Cald\'eron-Zygmund theory we have
\begin{equation}
\label{fancysf}
\|(\sum_{J\in \I_x} | \sum_{I\in \I_J} \epsilon_I \<f,h_I\> h_I(x)|_{H}^2)^{1/2}\|_{L^q(\R)}\le C_q \|f\|_{L^q(\R,H)}
\end{equation}
uniformly in all choices of signs $\epsilon_{I}\in\{-1,1\}$.
For $q=2$ this is an easy Hilbert space argument using orthogonality
of the functions $h_J$. For $q<2$ we use
a Cald\'eron-Zygmund decomposition of $|f|$ to obtain a weak endpoint at $q=1$
and then interpolate. For $q>2$ we use BMO techniques, i.e., we estimate
the sharp maximal function 
$$g^{\#}(x) =\sup_{x\in I} (\frac 1{|I|} \int_I (g-g_I)^2\,)^{1/2} = 
\sup_{x\in I} (\frac 1{|I|} \int_I g^2-g_I^2 \,)^{1/2}$$
of the function $g$ on the left hand side of (\ref{fancysf})
by the maximal function of $|f|$, 
and then use standard $L^q$ bounds for the sharp function and the maximal 
function.

Inequality (\ref{fancysf}) implies
$$\|(\sum_{1\le m\le \tilde{M}_\lambda(x)}
|\E(f|\D_{k_m(x)})(x)-\E(f|\D_{k_{m-1}(x)})(x)|_H^2)^{1/2}\|_q \le C_q \|f\|_{L^q(\R,H)}$$
and using that all jumps  are at least $\lambda/2$ proves the lemma.
\end{proof}

Define the $r$- variational norm of a sequence $g_k$ of elements in $H$
to be
$$\|g_k\|_{V^r(k)}:=
\sup_k|g_k|_{H}+ \sup_{M, k_0, k_1,\dots,k_M} (\sum_{m=1}^M |g_{k_m}-g_{k_{m-1}}|_H^r)^{1/r}$$
One may also define some ``weak'' variational norm
$$\|g_k\|_{V^{r,\infty}(k)}:=
\sup_k|g_k|_H+ \sup_{\lambda>0} \lambda M_{\lambda}^{1/r}.$$
where $M_\lambda$ is the maximal number of indices $k_0, k_1,\dots,k_M$
such that $|g_{k_m}-g_{k_{m-1}}|_{H}\ge \lambda$ for all $1\le m\le M$.
We have the usual estimate for the $V^r(k)$ norm in terms of $M_\lambda$
$$\|g_k\|_{V^r(k)}\le \|g_k\|_\infty+ 
C \int_0^\infty \lambda^{r} M_\lambda \, \frac{d\lambda}{\lambda}$$

The jump inequality in Lemma ~\ref{jumplemma} is almost a $V^{2,\infty}$ inequality, with the difference that in that inequality $\lambda$ is independent
of $x$, while in an honest $V^{2,\infty}$ inequality the parameter $\lambda$
may be maximized at every $x$ individually. Hence the jump inequality
is somewhat weaker than a $V^{2,\infty}$ inequality.
By integrating over all $\lambda$ and using Fubini one can abandon this
disadvantage of $\lambda$ being constant in $x$ and prove honest 
$V^r(k)$ norm estimates with $r>2$.

\begin{lemma}[Variational estimate]
\label{varlemma}
Let $1<q<\infty$ and $f\in L^q(\R,H)$. Then for $2<r<\infty$ we have
$$\|\|\E(f|\D_k)(x)\|_{V^r(k)} \|_{L_x^q}\le C_q(1+(r-2)^{-1}) \|f\|_{L^q(\R,H)}$$
where $C_q$ remains bounded on any compact interval of $(1,\infty)$.
\end{lemma}

\begin{proof}
For each $x$ and $\lambda>0$ we denote by $M_{\lambda}(x)$ the entropy number of the collection $\{\E(f|\D_k)(x):k\in\Z\}.$ We first consider this inequality for $|f|$ being the characteristic function of a set $A$.
Then $M_\lambda=0$ for $\lambda>1$.
Hence we can write for $2<r<\infty$ 
$$\|\E(f|\D_k)(x)\|_{V^r(k)}
\le C \left(\int_0^1  \lambda^2  M_\lambda(x) \, \, \lambda^{r-2} 
\frac{d\lambda}{\lambda}\right)^{1/r}$$
The right hand term is an $L^r_\lambda(d\mu)$ norm of $(\lambda^2 M_\lambda)^{1/r}(x)$ 
with respect to an appropriate measure space of total mass $\|\mu\|=\int_{0}^1\lambda^{r-3}d\lambda=(r-2)^{-1}$.

In the case $q=r$ we get
\begin{align*}
\|\|\E(f|\D_k)(x)\|_{V^r(k)}\|_{L_x^r}  
&\le  C \|\| (\lambda^2 M_\lambda(x)) ^{1/r}\|_{L^r_{\lambda}(d\mu)}\|_{L_x^r}
\\&= C \|\| (\lambda^2 M_\lambda(x)) ^{1/r}\|_{L_x^r}\|_{L^r_{\lambda}(d\mu)}
\\&\le C \| |A|^{1/r} \|_{L^r_{\lambda}(d\mu)}
\\&\le C (r-2)^{-1/r} |A|^{1/r}. 
\end{align*}
Here we have used that
$$\int  \lambda^2  M_\lambda(x)  \, dx \le C |A|$$
from the jump inequality in Lemma ~\ref{jumplemma} applied with $q=2$.
We remark that $(r-2)^{-1/r}$ is bounded by $1+(r-2)^{-1}$.

If $q>r$, then we invoke H\"older's inequality 
$$\| (\lambda^2 M_\lambda(x)) ^{1/r}\|_{L_{\lambda}^r(d\mu)} 
\le (r-2)^{1/q-1/r}\|(\lambda^2 M_\lambda(x)) ^{1/r}\|_{L^q(d\mu)} $$
and
$$\int  (\lambda^2  M_\lambda(x))^{q/r}  \, dx \le C_{2q/r} |A|$$
and then proceed as above to obtain
$$\|\|\E(f|\D_k)(x)\|_{V^r(k)}\|_{L_x^q} \le C_{2q/r} (r-2)^{-1/r} |A|^{1/q} 
$$
Observe that $2<2q/r<q$, so we can write $C_q$ instead of $C_{2q/r}$.

If $q<r$, we will prove a weak type inequality
$$m\{x: \|\E(f|\D_k)(x)\|_{V^r(k)}\ge \nu\} \le C_q (1+(r-2)^{-1}) \nu^{-q} \|f\|_{L^q(\R,H)}^q.$$
Define
$$E=\{x:\sup_{x\in I\in\D}\frac{1}{|I|}\int_{I}|f|(y)dy\ge \nu\}.$$

 Outside $E$, we may replace $f$ by the good part $g$ of the
Cald\'eron-Zygmund decomposition of $f$ in order to calculate the 
value of $\E(f|\D_k)$. As usual we have 
$$\|g\|_{L^r(\R,H)}\le C \nu^{1-q/r} \|f\|_{L^q(\R,H)}^{q/r}.$$
Hence we have
\begin{align*}
m\{x: \|\E(f|\D_k)(x)\|_{V^r(k)}\ge \nu\}&\le |E|+
m\{x\in E^c: \|\E(f|\D_k)(x)\|_{V^r(k)}\ge \nu\}
\\&\le C_q \nu^{-q} \|f\|_{L^q(\R,H)}^q+ C \nu^{-r} 
\|\|\E(f|\D_k)(x)\|_{V^r(k)} \|_{L_x^r(E^c)}^r
\\&\le C_q (1+(r-2)^{-1}) \nu^{-q} \|f\|_{L^q(\R,H)}^q
\end{align*}
The Lemma now follows by Marcinkiewicz interpolation, passing from 
restricted weak type to strong type inequalities.

\end{proof}

\section{General facts about  Walsh time-frequency analysis}
\label{generalfacts}
The endpoints of the dyadic intervals will be  called \emph{dyadic points}.
For each dyadic interval $\omega=[a,b]$, the subintervals $\omega_1:=[a,\frac{a+b}{2}]$ and  $\omega_2:=[\frac{a+b}{2},b]$ will be referred to as the left and right children of $\omega$, respectively.
\begin{definition}
For two tiles (or bitiles) $P$ and $P'$ we write $P\le P'$ if $I_P\subseteq I_{P'}$ and $\omega_{P'}\subseteq\omega_P$.
\end{definition}
\begin{definition}
A tree with top $(I_{\T},\xi_{\T})$ is a collection of bitiles $\T\subseteq\S$ such that $I_{P}\subseteq I_{\T}$ and $\xi_{\T}\in\omega_{P}$ for each $P\in \T$.  An i-tree is a tree $\T$ such that $\xi_{\T}\in\omega_{P,i}$ for each $P\in \T$.
\end{definition}

\begin{definition}
Fix some $f:\R_+\to\R$.
For a  finite subset of bitiles $\S'\subseteq\S$ define its \emph{size} relative to $f$ as 
$$\size(\S'):=\sup_{\T}\left(\frac{1}{|I_{\T}|}\sum_{P\in \T}|\langle f, w_{P_1}\rangle|^2\right)^{\frac{1}{2}}$$ where the supremum is taken over all the $2$-trees $\T \subset \S'$. 
\end{definition}

We recall a few important results regarding the size.
\begin{proposition}
\label{g11a00uugg55}
For each $1<s<\infty$, each 2-tree $\T$  and  each $f\in L^s(\R_+)$ we have 
$$\left(\frac1{|I_\T|}{\sum_{P\in\T}|\langle f,w_{P_1}\rangle|^2}\right)^{1/2}\lesssim \inf_{x\in I_\T}M_s f(x).$$
\end{proposition}
\begin{proof}
See for example Lemma 1.8.1 in \cite{Th}.
\end{proof} 

The following Bessel type inequality, see for example \cite{MTT6}, will be used to organize  collections of bitiles into trees. 
\begin{proposition}
\label{Besselsineq}
Let $\S'\subseteq\S$ be a collection  of tiles and define $\Delta:=[-\log_2(\size(\S'))]$, where the size is understood with respect to some function $f\in L^2(\R_+)$. Then $\S'$ can be written as a disjoint union $\S'=\bigcup_{n\ge \Delta}\P_n,$ where $\size(\P_n)\le 2^{-n}$ and each $\P_n$ consists of a family $\F_{\P_n}$ of pairwise disjoint trees  satisfying 
\begin{equation}
\label{e:intp}
\sum_{\T\in\F_{\P_n}}|I_\T|\lesssim 2^{2n}\|f\|_2^2,
\end{equation} 
with bounds independent of $\S'$, $n$ and $f$. 
\end{proposition}

Elementary computations show that for each tile $P=[2^in,2^i(n+1))\times[2^{-i}l,2^{-i}(l+1))$, each $l'\ge 0$ and each $\xi \in [2^{-i}l',2^{-i}(l'+1))$ we have  
$$w_{P}(x)=1_{I_{P}}(x)e(2^{-i}l\otimes x)$$ 
$$w_{P}(x)e(\xi\otimes x)=\epsilon(P,\xi)1_{I_{P}}(x)e(2^{-i}|l'-l|\otimes x)$$ 
where $\epsilon(P,\xi)\in\{-1,1\}$ depends on $P$ and $\xi$ but not on $x$. In particular, if $\T$ is a 2-tree and $P\in\T$ then
$$w_{P_1}(x)e(\xi_{\T}\otimes x)=\epsilon(P,\xi_{\T})w_{P'}(x)$$
where $P'=[2^in,2^i(n+1))\times[2^{-i},2^{-i+1})$,
and thus $w_{P_1}(x)e(\xi_{\T}\otimes x)$ is constant on both the left half and the right half of $I_P$. 
An immediate consequence is that for each $k\in\Z$ and each $a_{P}\in\R$
$$e(\xi_{\T}\otimes x)\sum_{P\in\T:|I_P|\ge 2^k}a_{P}w_{P_1}(x)=\E(e(\xi_{\T}\otimes \cdot)\sum_{P\in\T}a_{P}w_{P_1}|\D_{k-1})(x).$$
Since 
\begin{align*}
w_{P_1}(x)e(\xi_{\T}\otimes x)&=\epsilon(P,\xi_{\T})2^{-i/2}W_1(2^{-i}x-n)\\&=\epsilon(P,\xi_{\T})h_{I_P}(x)
\end{align*}
where $h$ is the Haar function, the classical theory of wavelets and John-Nirenberg's inequality imply the following.
\begin{theorem}
\label{BMOestimy}
Let  $\T$ be a 2-tree and assume $(a_{P})_{P\in\T}\in\R$ satisfy
$$(\frac1{|I|}\sum_{P\in \T\atop{I_{P}\subseteq I}}|a_{I_P}|^2)^{1/2}\le B,$$ 
for each dyadic interval $I$. Then for each $1<s<\infty$
$$\|e(\xi_{\T}\otimes \cdot)\sum_{P\in\T}a_{P}w_{P_1}\|_{BMO}\lesssim B$$
and
$$\|e(\xi_{\T}\otimes \cdot)\sum_{P\in\T}a_{P}w_{P_1}\|_{s}\lesssim_s B|I_{\T}|^{1/s}.$$
\end{theorem}

As an immediate  consequence of Theorem ~\ref{BMOestimy} and of Lemma ~\ref{varlemma} we obtain the following.

\begin{theorem}
\label{BMOestimy1}
Let  $\T$ be a 2-tree, $f:\R_+\to\R$ and let $\size(\T)$ denote the size of $\T$ with respect to the function $f$. 
Then for each $1<s<\infty$
$$\|\|\sum_{P\in\T\atop{|I_P|\le 2^k}}a_{P}w_{P_1}(x)\|_{V^r(k)}\|_{L_x^s}\lesssim_s \size(\T)|I_{\T}|^{1/s}.$$
\end{theorem}

\section{A generalization of a Lemma of Bourgain}
\label{Bourgain}
In this section we generalize a maximal multiplier result due to Bourgain \cite{Bo1}
We begin with the following easy consequence of Minkowski's inequality.
\begin{lemma}
\label{vartwoterms}
Let $\Xi$ be a finite set. Consider also two sequences $a_k$ and $b_k$ in the Hilbert space $l^{2}(\Xi)$ and define $a_k\star b_k\in l^{2}(\Xi)$ by $(a_k\star b_k)_{\xi}=(a_{k})_{\xi}(b_{k})_{\xi}$. Then
$$\|a_k\star b_k\|_{V^r(k)}\lesssim (\sum_{\xi\in\Xi}\|(a_{k})_{\xi}\|_{V^r(k)}^2\|(b_{k})_{\xi}\|_{V^r(k)}^2)^{1/2},$$
\end{lemma}

\begin{proposition}
\label{entropycorollary}
Let $H$ be a Hilbert space. Assume we are given
a set $A$ of linear functionals $f\to f^{(\alpha)}=\<f,e^{(\alpha)}\>$, $e^{(\alpha)}\in H$, of norm less than $\epsilon$ such that 
$$\sum_{\alpha\in A} |f^{(\alpha)}|^2 \le |f|^2$$
for each $f\in H$. Set $N=\epsilon^2 |A|$.
Let $f_k$ be a sequence of $H$-valued 
functions on $\R$ such that we have the variational inequality
$$\| \|f_k(x)\|_{V^r(k)} \|_{L_x^2}\le F.$$
Then we have
$$\| \| \sup_k | f_k^{(\alpha)}(x)|\|_{l^2(A)} \|_{L_x^2}\le 
C N^{r/4-1/2} F.$$
\end{proposition}

A special example of a collection of linear functionals as in the
Lemma can be obtained by choosing the $e^{(\alpha)}$ to be an orthonormal
family of vectors and $\epsilon=1$. Our main application will involve a 
more general set of linear functionals.
We remark that  the difficulty in this proposition comes 
from the fact that we take the supremum in $k$ before we take the square sum of the components.

\begin{proof}
Fix $x$ and define $C_x=\{f_k(x)\}$ and $d(x)=\diam (C_x).$  
It suffices to prove the Proposition in the case  $C_x$ is finite and then to invoke the Monotone Convergence Theorem. Also, we can assume with no loss of generality that $C_x$ contains the origin ${\bf 0}$. For each $\lambda>0$ denote by $N_{\lambda}(x)$ the minimum number of balls with radius $\lambda$ and centered at elements of $C_x$, whose union covers $C_x.$ It is an easy exercise to prove that
\begin{equation}
\label{2uanoap}
\sup_{\lambda>0}\lambda N_{\lambda}^{1/r}(x)\lesssim_{r} \|f_k(x)\|_{V^r(k)},
\end{equation}
with the implicit constant depending only on $r$.
 For each $n\ge -\log_2(d(x))$, let $C_{n,x}$ be a  collection of elements of $(C_x-C_x)$ such that
\begin{align*}
|c|_{H}&\le 2^{-n+2}\;\;\hbox{for each}\;c\in C_{n,x},\\
\sharp C_n&\le N_{2^{-n}}(x)+1
\end{align*}
and each $c\in C_{n,x}$ can be written as 
\begin{equation}
\label{sumrepresent}
c=\sum_{n\ge -\log_2(d(x))}c_n\;\;\hbox{with}\;c_n\in C_{n,x}.
\end{equation}
Here is how $C_{n,x}$ is constructed. For each $n\ge -\log_2(d(x))$ define $B_{n,x}$ to be a collection of $N_{2^{-n}}(x)$ elements of $C_x$ such that the balls with centers in $B_{n,x}$ and radius $2^{-n}$ cover $C_x$. If $n=[-\log_2(d(x))]-1$ define $B_{n,x}=\{{\bf 0}\}$. For each $n\ge -\log_2(d(x))$ and each $c\in B_{n,x}$, choose an element $c'\in B_{n-1,x}$ such that the ball centered at $c$ and with radius $2^{-n}$ intersects  the ball centered at $c'$ and with radius $2^{-n+1}$. Define 
$$C_{n,x}:=\{c-c':c\in B_{n,x}\}\cup\{{\bf 0}\}.$$
Since $C_x$ is finite, for each $c\in C_x$ there is $n$ such that $c\in B_{n,x}$. To verify the representation ~\eqref{sumrepresent} for an arbitrary $c\in C_x$,  denote as above by $c'$ the element from $B_{n-1,x}$ associated with $c$, by $c''$ the element from $B_{n-2,x}$ associated with $c'$ and so on, and note that this sequence will eventually terminate with ${\bf 0}$. Hence we can write
$$c=(c-c')+(c'-c'')+\ldots .$$ Note also that by construction, each element of $C_{n,x}$ has norm at most $2^{-n+2}$.

This together with  inequality ~\eqref{2uanoap}  further allows us to write for each $x$ and $\alpha$
$$\sup_k |f_k^{(\alpha)}(x)|$$
$$\le \sum_{n\ge -\log_2(d(x))} \sup_{c_n\in C_{n,x}} |c_n^{(\alpha)}|$$
$$\lesssim \sum_{n\ge -\log_2(d(x))}\min\left(2^{-n}\epsilon,(\sum_{c_n\in C_{n,x}} |c_n^{(\alpha)}|^2)^{1/2}\right).$$
Summing over $\alpha$ we get
$$\sum_\alpha (\sup_k |f_k^{(\alpha)}(x)|)^2 $$
$$\lesssim \sum_{n\ge-\log_2(d(x))} \min(2^{-2n} N, \sum_{c_n\in C_{n,x}}|c_n|_{H}^2)$$
$$\lesssim  2^{-2n}\sum_{n\ge -\log_2(d(x))}\min(N, N_{2^{-n}}(x)).$$
Taking finally the $L^2$ norm in $x$ gives
$$\| \| \sup_k | f_k^{(\alpha)}(x) | \|_{l^2(A)} \|_{L_x^2}^2$$
$$\lesssim  \int \sum_{2^{-n}<d(x)/{N^{1/2}}} 2^{-2n} N \,dx +
\int\sum_{d(x)/N^{1/2}\le 2^{-n}\le d(x)} 2^{-2n} N_{2^{-n}}(x) \, dx$$
$$\lesssim \int d^2(x)\, dx+ \int\sum_{d(x)/N^{1/2}\le 2^{-n}} 2^{-(2-r)n}2^{-rn} N_{2^{-n}}(x)\,dx$$
$$\lesssim \int d^2(x)\, dx+ N^{r/2-1}\int d(x)^{2-r}\sum_{n} 2^{-rn} N_{2^{-n}}(x)\, dx$$
$$\lesssim N^{r/2-1}\int \|f_k(x)\|_{V^r(k)}^2\, dx$$
$$\lesssim N^{r/2-1}F^2.$$
This finishes the proof.
\end{proof}

\begin{corollary}
\label{weightedbourgain}
Let $2<r<\infty$.
Assume we are given a set $\Xi\subset \R_+$ of cardinality $N>1$
and assume that there is no dyadic interval of length $1$
which contains more than one point in $\Xi$.
For every $k\ge 0$ define $\Omega_k$ to be the union
of all dyadic intervals of length $2^{-k}$ which have nonempty intersection
with $\Xi$. For each $\omega\in \Omega_k$ let $\epsilon_\omega$ be a number
so that for every nested sequence of intervals $\omega_k\in \Omega_k$ 
we have 
\begin{equation}
\label{Varesteps}
\|\epsilon_{\omega_k}\|_{V^r(k)}\le \sigma.
\end{equation}
Define $$\Delta_kf(x)=(\sum_{\omega\in \Omega_k} \epsilon_\omega 
1_{\omega} \widehat{f})\check{\ }(x).$$
Then 
$$\|\sup_{k\ge 0}|\Delta_kf|\|_2\lesssim_r \sigma N^{r/4-1/2} \|f\|_2.$$
\end{corollary}

\begin{proof}
Fix $f\in L^{2}(\R_{+})$.
For each $k\ge 0$ we will denote by $\omega_{\xi,k}$ the unique dyadic interval in $\Omega_k$ such that $\xi\in\omega_{\xi,k}$, and by $w_\xi(x)=e(x\otimes \xi)$. Let $H$ be the $N$ dimensional Hilbert space $l^2(\Xi)$. Define the sequence of functions $f_k:\R\to H$, $k\ge 0$, by
\begin{align*}
(f_k(x))_{\xi}&=\epsilon_{\omega_{\xi,k}}(\widehat{f}1_{\omega_{\xi,k}})\check{\ }(x)\\&=\epsilon_{\omega_{\xi,k}}w_\xi(x)\E(fw_{\xi}|\D_{k})(x),
\end{align*}
and note that
\begin{equation}
\label{halkiui863}
(f_k(x))_{\xi}=w_{\xi}(y)(f_k(x\oplus y))_{\xi}
\end{equation}
for all $y\in [0,1)$.

To construct the vectors $e^{(\alpha)}$, 
choose some small negative integer $m$ so that all $w_\xi$ are constant on 
dyadic subintervals of $[0,1)$ of length
$2^{m}$. We write $w_\xi(J)$ for this constant value
on such an interval $J$. For each such interval $J_{\alpha}$, $\alpha\in A:=\{1,2,\ldots,2^{-m}\}$, define
$$e^{(\alpha)}=(2^{m/2}w_\xi(J_\alpha))_{\xi\in\Xi}.$$
The corresponding linear functionals are of norm $\epsilon= 2^{m/2} |\Xi|^{1/2}$. 
We also have $$\sum_\alpha  |g^{(\alpha)}|^2\le 
\int_0^1 |\sum_{\xi\in \Xi} g_\xi w_\xi(x)|^2 \, dx
\le \sum_\xi |g_\xi|^2,$$
for each $g\in H.$ In the last inequality we have used that the functions $w_\xi$ are orthogonal on $[0,1)$.
Hence the functionals satisfy the assumption of Proposition \ref{entropycorollary} with $N=\epsilon^2|A|=|\Xi|$.
We observe the following
\begin{align*}
\|\| \sup_k | f_k^{(\alpha)}(x)|\|_{l^2(A)} \|_{L_x^2}&=\int_{\R_+}\int_0^1 \sup_k|\sum_{\xi\in\Xi}\epsilon_{\omega_{\xi,k}}w_{\xi}(y)(\widehat{f}1_{\omega_{\xi,k}})\check{\ }(x)|^2dy\,dx\\&=\int_{\R_+}\int_0^1 \sup_k|\sum_{\xi\in\Xi}\epsilon_{\omega_{\xi,k}}w_{\xi}(y)(\widehat{f}1_{\omega_{\xi,k}})\check{\ }(x\oplus y)|^2dy\,dx\\&=\int_{\R_+} \sup_k|\sum_{\xi\in\Xi}\epsilon_{\omega_{\xi,k}}(\widehat{f}1_{\omega_{\xi,k}})\check{\ }(x)|^2\,dx\\&=\|\sup_{k\ge 0}|\Delta_kf|\|_2^2
\end{align*}
where the last equality is a consequence of ~\eqref{halkiui863}.
The corollary now follows from Proposition ~\ref{entropycorollary} once we verify that
$$\| \|f_k(x)\|_{V^r(k)} \|_{L_x^2}\lesssim_r \sigma\|f\|_{2}.$$
Note that for each $x$, $f_k(x)=a_{k,x}\star b_{k,x}$, where $(a_{k,x})_{\xi}=\epsilon_{\omega_{\xi,k}}$ and $(a_{k,x})_{\xi}=w_\xi(x)\E(fw_{\xi}|\D_{k})(x)$. The above estimate is now a consequence of Lemma ~\ref{varlemma}, Lemma ~\ref{vartwoterms} and inequality ~\eqref{Varesteps}.
\end{proof}

An argument very similar to the above also proves the following version of Corollary ~\ref{weightedbourgain}:
\begin{corollary}
\label{ttzitzijj6}
Consider a collection $\Omega$ of $N$ disjoint dyadic  intervals $\omega\in\R_+$. For each $\omega\in \Omega$ and each $k\in \Z$ let $\epsilon_{k,\omega}\in\R$ . Define 
$$\Delta_kf(x):=\sum_{\omega\in \Omega}\epsilon_{k,\omega}(\widehat{f}1_{\omega})\check{\ }(x).$$ 
Then for each $r>2$
$$\|\sup_k|\Delta_kf|\|_{L^2}\lesssim_r N^{r/4-1/2} \sup_{\omega\in \Omega}\|\epsilon_{k,\omega}\|_{V^r(k)}\|f\|_{L^2}.$$
\end{corollary}

It turns out that the results of corollaries ~\ref{weightedbourgain} and  ~\ref{ttzitzijj6} are not general enough for our applications, and so we prove the following more general version. Consider now an arbitrary set $\Xi=\{\xi_1,\ldots,\xi_N\}$ with no further restrictions on it, and for each $k\in \Z$ define $\Omega_k$ to be the set of all dyadic intervals of length $2^{-k}$ which contain some element of $\Xi$. We now associate to each $\omega\in \bigcup_k \Omega_k$ a number $\epsilon_\omega\in\R$ and define
\begin{equation}
\label{Bourgeneral}
\Delta_kf(x):=\sum_{\omega\in \Omega_k}\epsilon_{\omega}(\widehat{f}1_{\omega})\check{\ }(x).
\end{equation}

\begin{proposition}
\label{prop:gh098cfhh}
For each $r>2$ we have the inequality
$$\|\sup_{k}|\Delta_kf|\|_{2} \lesssim_r  N^{r/4-1/2} \sigma\|f\|_2,$$
where
$$\sigma=\sup_{n}\sup_{\xi_n\in\omega_k\in \Omega_k}\|\epsilon_{\omega_k}\|_{V^r(k)}.$$
\end{proposition}
\begin{proof}
It suffices as before to assume that the index $k$ runs through a finite interval $\{a,a+1,\ldots,b\}$ with $a,b\in\Z$.
We can find a sequence  $a=k_0<k_1<\ldots<k_L=b$ with $L\le N$, such that for each $0\le j\le L-1$, $\Omega_k$ has the same cardinality when $k_j\le k<k_{j+1}$. If $\widehat{f_j}:=(\sum_{\omega\in \Omega_{k_j}}1_{\omega}-\sum_{\omega\in \Omega_{k_{j+1}}}1_{\omega})\widehat{f}$, then the functions $f_j$ are pairwise orthogonal. 
We can now bound $\|\sup_k|\Delta_kf|\|_{2}$ by 

\begin{equation}
\label{-1-3-5-7ii}
\|\sup_{j}\sup_{k_j\le k<k_{j+1}}|(\sum_{\omega\in \Omega_{k_{j+1}}}\epsilon_{{\omega}(k)}1_{\omega}\sum_{j'>j}\widehat{f}_{j'})\check{\ } |\|_{2}+
\end{equation} 
\begin{equation}
\label{-1-3-5-7iii}
+\|\sup_{j}\sup_{k_j\le k<k_{j+1}}|(\sum_{\omega\in \Omega_{k}}\epsilon_{\omega}1_{\omega}\widehat{f}_{j})\check{\ } |\|_{2}.
\end{equation}
For each $\omega\in \Omega_{k_{j+1}}$ and each $k_j\le k<k_{j+1}$, ${\omega}(k)$ is defined to be the interval in $\Omega_k$ containing $\omega.$ Corollary ~\ref{weightedbourgain} and scaling invariance show that the term ~\eqref{-1-3-5-7iii} can be bounded by 
\begin{align*}
(\sum_{j}\|\sup_{k_j\le k<k_{j+1}}|(\sum_{\omega\in \Omega_{k}}\epsilon_{\omega}1_{\omega}\widehat{f}_{j})\check{\ } |\|_{2}^2)^{1/2}&\lesssim (\sum_{j}N^{r/2-1}\sup_{n}\sup_{\xi_n\in\omega_k\in \Omega_k\atop{k_j\le k<k_{j+1}}}\|\epsilon_{\omega_k}\|_{V^r(k)}^2\|f_j\|_{2}^2)^{1/2}\\&\lesssim \sigma N^{r/4-1/2}\|f\|_{2}.
\end{align*}
To estimate the term in ~\eqref{-1-3-5-7ii}, define the maximal operators $$O_j^{*}(h) :=\sup_{k_j\le k<k_{j+1}}|(\sum_{\omega\in \Omega_{k_{j+1}}}\epsilon_{{\omega}(k)}1_{\omega}\widehat{h})\check{\ } |. $$ 
We will  argue that
$$\|\sup_{1\le j\le N}O_{j}^{*}(\sum_{j\le j'\le N}f_{j'}) \|_{2}
\lesssim \sigma N^{r/4-1/2}(\sum_{j=1}^L\|f_j\|_{2}^2)^{1/2}.$$ It suffices to consider only dyadic values of $N$ so we will assume that $N=2^{M}$, for some $M\ge 0$.
For each $0\le m\le M$, denote by  $A_m$ the best constant for which the following inequality holds for all discrete dyadic intervals $J=(j_1,j_2]:=\{j_1+1,j_1+2,\ldots,j_{2}\}\footnote{$j_1$ and $j_2$ are of the form $a2^b$ with  $a,b\in\Z_{+}$}\subseteq \{1,2\ldots,2^M\}$ with $2^m$ elements

$$\|\sup_{j\in J}O_{j}^{*}(\sum_{j\le j'\le j_2}f_{j'}) \|_{2}\lesssim A_m(\sum_{j\in J}\|f_j\|_{2}^2)^{1/2}.$$ We will use a reasoning  similar to the one  in the proof of the Rademacher-Menshov inequality, to argue that $A_{M}\lesssim B_{M}$, where
$$B_{m}:= \sigma 2^{m(r/4-1/2)}.$$  We can  write  for each $0\le m\le M-1$ and each discrete dyadic interval $J=(j_1,j_2]\subseteq\{1,2,\ldots,2^{M}\}$ having $2^{m+1}$ elements and midpoint $j_3:=j_1+2^m$ 
\begin{equation*}
\|\sup_{j\in J}O_{j}^{*}(\sum_{j\le j'\le j_2}f_{j'}) \|_{2}^2
 \end{equation*}
\begin{equation*}
\le
\|\sup_{j_3+1\le j\le j_2}O_{j}^{*}(\sum_{j\le j'\le j_2}f_{j'}) \|_{2}^2+
\end{equation*}
\begin{equation*}
+\left(\|\sup_{j_1+1\le j\le j_3}O_{j}^{*}(\sum_{j \le j'\le j_3}f_{j'}) \|_{2}+\|\sup_{j_1+1\le j\le j_3}O_{j}^{*}(\sum_{j_3+1\le j'\le j_2}f_{j'}) \|_{2}\right)^2.
\end{equation*}
We then use the definition of $A_m$ for the first two terms above and Corollary ~\ref{ttzitzijj6} for the third one, to bound the sum above by 
\begin{equation*}
 A_{m}^2\sum_{j_3+1\le j'\le j_2}\|f_{j'}\|_{2}^2+ (A_{m}(\sum_{j_1\le j'\le j_3}\|f_{j'}\|_{2}^2)^{1/2}+CB_{m}(\sum_{j_3+1\le j'\le j_2}\|f_{j'}\|_{2}^2)^{1/2})^{2}
\end{equation*}
\begin{equation*}
\le (A_{m}+CB_{m})^2\sum_{j\in J}\|f_{j}\|_{2}^2.
\end{equation*}
We conclude that $A_{{m+1}}\le A_{m}+CB_{m}$ for each $0\le m \le M-1$, which together with the fact that $A_0=0$ proves that $A_{M}\lesssim B_{M}.$
\end{proof}

\begin{remark}
If in the above proposition we choose $\epsilon_{\omega}=1$ for each $\omega$, we recover the result of Bourgain from \cite{Bo1}, with a slightly larger dependence on $N$ of the bound. While Bourgain's bound is logarithmic in $N$, a bound of the form $N^{r/4-1/2}$ will suffice for our later applications, since we afford to take $r$ as close to 2 as we want.
\end{remark}

\section{Pointwise estimates outside exceptional sets}
\label{treeestimates}

\subsection{An estimate for a collection of 2-trees}
\label{type2tree}

Assume we have a collection $\S'\subseteq\S$ of bitiles which can be written as a not necessarily disjoint union of 2-trees 
$$\S'=\bigcup_{\T\in \F} \T.$$
We  shall assume that if $\T\in \F$, then $\T$ is indeed the maximal 2-tree in $\S'$ with the top  $(I_{\T},\xi_{\T})$, that is, all bitiles
in $P\in\S'$ which satisfy $I_{P}\subseteq I_{\T}$ and $\xi_{\T}\in \omega_{P,2}$ are in $\T$.

\begin{theorem}
\label{mainineqyy}
For each  $\beta\ge 1$, $\gamma>0$ and each $(a_P)_{P\in\S'}$ define the  exceptional sets
$$E^{(1)}=\{x:\sum_{\T\in\F}1_{I_{\T}}(x)>\beta\},$$
$$E^{(2)}=\bigcup_{\T\in\F}\{x:\|\sum_{P\in\T\atop{|I_P|< 2^{k}}}a_{P}w_{P_1}(x)\|_{V^r(k)}>\gamma\}.$$
Then  for each $x\notin E^{(1)}\cup E^{(2)}$ and each $r>2$ we have 
$$\|(\sum_{P\in\S'\atop{|I_P|\le 2^{k}}}a_{P}w_{P_1}(x)1_{\omega_{P,2}}(\theta))_{k\in\Z}\|_{M_2^*(\theta)}\lesssim_{r} \gamma\beta^{r/4-1/2}.$$
\end{theorem}

\begin{proof}
Fix $x$ not in the union of the exceptional sets  and let $\F_x$ be the family of all trees $\T\in\F$ with $x\in I_{\T}$. Define 
$$\Xi_{x}=\{\xi_{\T},\T\in\F_{x}\}.$$
For each $k\in\Z$ let $\Omega_{k}$ be the collection of dyadic frequency intervals of length $2^{-k}$  which contain an element of $\Xi_{x}$.
Let $\tilde{\Omega}_{k}$ be the collection of all children of
intervals in $\Omega_{k-1}$ that are not themselves in $\Omega_{k}$.
Observe that both $\bigcup_{k'} \tilde{\Omega}_{k'}$ and $\Omega_{k}\cup \bigcup_{k'\le k}\tilde{\Omega}_{k'}$ are  collections of pairwise disjoint intervals which cover $\R_+$ (with the possible exception of finitely many dyadic points). Moreover we can write 
$$ 
\sum_{P\in\S' :|I_P|< 2^k} a_Pw_{P_1}(x)1_{\omega_{P,2}}(\theta)= 
$$
\begin{equation*}
=\sum_{\omega \in \Omega_{k}} 1_{\omega}(\theta)
\sum_{P\in\S'\atop{|I_P|< 2^k,\;\omega\cap \omega_{s,2}\not=\emptyset}}a_P w_{P_1}(x)1_{\omega_{P,2}}(\theta)
\end{equation*}

\begin{equation*}
+
\sum_{k'\le k} 
\sum_{\omega \in \tilde{\Omega}_{k'}}
1_{\omega}(\theta) \sum_{P\in\S'\atop{|I_P|< 2^{k},\;\omega \cap \omega_{P,2}\not=\emptyset}}a_P w_{P_1}(x)1_{\omega_{P,2}}(\theta).
\end{equation*}
Indeed, if $1_\omega(\theta)w_{P_1}1_{\omega_{P,2}}(\theta)\not\equiv 0$ for some $\omega\in \Omega_{k}\cup \bigcup_{k'\le k}\tilde{\Omega}_{k'}$ and $P\in \S'$, then this implies that $\omega\cap\omega_{P,2}\not=\emptyset$. Moreover, when $\omega \in  \Omega_{k}$, this latter restriction together with  $|I_P|<2^{k}$ is equivalent with just asking that $\omega\subseteq \omega_{P,2}$. Similarly, when $\omega \in\bigcup_{k'\le k}\tilde{\Omega}_{k'}$ then $\omega_{P,2}\subsetneq \omega$ is impossible, which in turn makes the requirement  $|I_P|<2^{k}$ superfluous. Indeed  $\omega_{P,2}\subsetneq \omega$ would imply that $\omega_{P}\subseteq \omega$, contradicting the fact that  $\omega_P$ contains an element from $\Xi_{x}$ while   $\omega$ does not. Hence we can rewrite 
  
$$ 
\sum_{P\in\S' :|I_P|< 2^k} a_P w_{P_1}(x)1_{\omega_{P,2}}(\theta)=
$$
\begin{equation}
\label{mktrees1}
=\sum_{\omega \in \Omega_{k}} 1_{\omega}(\theta)
\sum_{P\in\S'\atop{\omega\subseteq \omega_{P,2}}}a_P w_{P_1}(x)1_{\omega_{P,2}}(\theta)
\end{equation}

\begin{equation}
\label{mktrees2}
+
\sum_{k'\le k} 
\sum_{\omega \in \tilde{\Omega}_{k'}}
1_{\omega}(\theta) \sum_{P\in\S'\atop{\omega\subseteq \omega_{P,2}}}
a_P w_{P_1}(x)1_{\omega_{P,2}}(\theta).
\end{equation}

The  multiplier in  ~\eqref{mktrees2} can be written more conveniently as  
$$(1-\sum_{\tilde{\omega}\in \Omega_{k}}1_{\tilde{\omega}})\left(\sum_{k'} 
\sum_{\omega \in \tilde{\Omega}_{k'}}
1_{\omega}(\theta) \sum_{P\in\S',\;\omega \subseteq\omega_{P,2}}
a_P w_{P_1}(x)1_{\omega_{P,2}}(\theta)\right)=$$ 
$$=(1-\sum_{\tilde{\omega}\in \Omega_{k}}1_{\tilde{\omega}})\sum_{P\in\S'}
a_P w_{P_1}(x)1_{\omega_{P,2}}(\theta),$$
given the fact that $(\bigcup_{I\in\Omega_{k}} I)^{c}=\bigcup_{k'\le k}\bigcup_{I\in\tilde{\Omega}_{k'}}I$ and $(\bigcup_{k'\le k}\tilde{\Omega}_{k'})\bigcap(\bigcup_{k'> k}\tilde{\Omega}_{k'})=\emptyset,$ modulo some dyadic points.
This maximal multiplier operator is now easily seen to be the composition of two operators. One is the identity minus Bourgain's maximal operator for which Proposition ~\ref{prop:gh098cfhh} provides good bounds. The second one is a linear operator associated with the multiplier $\sum_{P\in\S'}a_P w_{P_1}(x)1_{\omega_{P,2}}(\theta)$. To analyze the latter operator, we note that for each $\theta$ the contribution to the multiplier comes from a single tree. To see this note that the collection   

$$\X:=\{P\in\S':x\in I_{P},\; \theta\in \omega_{P,2}\}$$
is finite and totally ordered and hence it contains a maximum element $P_{\theta}$. If $\T_{\theta}\in\F$ is one of the 2-trees to which $P_{\theta}$ belongs, then by the maximality condition in the hypothesis it follows that $P\in\T_\theta$ for each $P\in\X$. Moreover, there is some $k$ such that
$$\sum_{P\in\S'}a_P w_{P_1}(x)1_{\omega_{P,2}}(\theta)=\sum_{P\in\T_{\theta}\atop{|I_P|\le 2^{k}}}a_P w_{P_1}(x).$$ 
By invoking  Proposition ~\ref{prop:gh098cfhh} and the fact that $x\notin E^{(1)}\cup E^{(2)}$ we get that 
$$
\left\|\left((1-\sum_{\tilde{\omega}\in \Omega_{k}}1_{\tilde{\omega}})\sum_{P\in\S'}a_P w_{P_1}(x)1_{\omega_{P,2}}\right)_{k\in\Z}\right\|_{M_2^{*}}$$ 
$$\le(1+\|(\sum_{\tilde{\omega}\in \Omega_{k}}1_{\tilde{\omega}})_{k\in\Z}\|_{M_2^{*}})\|\sum_{P\in\S'}a_P w_{P_1}(x)1_{\omega_{P,2}}(\theta)\|_{L^{\infty}(\theta)}$$ 
$$\lesssim\gamma\beta^{r/4-1/2}.$$

The term \eqref{mktrees1} is clearly of the form
$$
\sum_{\omega \in \Omega_k} 1_{\omega}(\theta)\epsilon_\omega 
$$
with
$$\epsilon_\omega=\sum_{P\in\S'\atop{\omega\subseteq \omega_{P,2}}}a_P w_{P_1}(x).$$
We claim that for each nested sequence of intervals $\omega_k\in \Omega_k$
 we have the variational norm estimate
\begin{equation}
\label{ff55rt78}
\|\epsilon_{\omega_k}\|_{V^r(k)}\lesssim_r \gamma.
\end{equation}
This follows immediately from the fact that all bitiles contributing to
$\epsilon_{\omega_k}$
belong to a single tree $\T\in \F$. Indeed, the collection 
$$\O:=\{P\in\S':x\in I_{P}, \;\omega_k\subseteq \omega_{P,2}\;\hbox{for some \;}k\}$$
is finite and totally ordered, so it has a maximum element $P_x$. If $\T_x\in\F_x$ is one of the 2-trees to which $P_x$ belongs, then from the maximality condition in the hypothesis it follows that $P\in\T_x$ for each $P\in\O$. Moreover, for each $k$
$$\{P\in\S':\omega_k\subseteq \omega_{P,2}\}=\{P\in\T_x: |I_P|\le 2^k\}.$$   
Thus ~\eqref{ff55rt78} and  Proposition ~\ref{prop:gh098cfhh} imply that 
$$\left\|\left(\sum_{\omega \in \Omega_{k}} 1_{\omega}(\theta)
\sum_{P\in\S'\atop{\omega\subseteq \omega_{P,2}}}a_P w_{P_1}(x)1_{\omega_{P,2}}(\theta)\right)_{k\in\Z}\right\|_{M_2^{*}(\theta)}\lesssim\gamma\beta^{r/4-1/2}.$$
\end{proof}

\subsection{An estimate for a collection of 1-trees}
\label{type1tree}

The discussion here is very similar to that for 2-trees.
Assume we have a collection $\S'$ of bitiles which can be written as a
not necessarily disjoint union of finitely many 1-trees 
$$\S'=\bigcup_{\T\in \F}\T.$$
We shall assume that for every $P\in\S'$  there does not exist a tree $\T\in \F$ with
$I_{P}\subset I_{\T}$ and $\xi_{\T}\in \omega_{P,2}$.
This assumption does in particular imply that the upper tiles
$P_2$ are pairwise disjoint. For assume not and $I_{P}\subsetneq I_{P'}$
and $\omega_{P',2}\subsetneq \omega_{P,2}$ 
for some $P,P'$, then it is easy to see that the upper tile $P_2$ 
violates the above assumption with respect to any tree to which $P'$ belongs.

\begin{theorem}
\label{mainineqyy12}
Let $(a_P)_{P\in\S'}$ satisfy 
\begin{equation}
\label{numaiastaggg}
\sup_{P\in\S'}\frac{|a_P|}{|I_P|^{1/2}}\le \sigma.
\end{equation}
For each  $\alpha\ge 1$  define the  exceptional set
$$E=\{x:\sum_{\T\in\F}1_{I_{\T}}(x)>\beta\}.$$
Then  for each $x\notin E$ and each $r>2$ we have 
$$\|(\sum_{P\in\S'\atop{|I_P|\le 2^{k}}}a_{P}w_{P_1}(x)1_{\omega_{P,2}})_{k\in\Z}\|_{M_2^*}\lesssim_{r} \sigma\beta^{r/4-1/2}.$$
\end{theorem}
\begin{proof}
As before we write
$$ 
\sum_{P\in\S' :|I_P|< 2^k} a_P w_{P_1}(x)1_{\omega_{P,2}}(\theta)=
$$
\begin{equation*}
=\sum_{\omega \in \Omega_{k}} 1_{\omega}(\theta)
\sum_{P\in\S'\atop{\omega\subseteq \omega_{P,2}}}a_P w_{P_1}(x)1_{\omega_{P,2}}(\theta)
\end{equation*}

\begin{equation*}
+
\sum_{k'\le k*} 
\sum_{\omega \in \tilde{\Omega}_{k'}}
1_{\omega}(\theta) \sum_{P\in\S'\atop{\omega\subseteq \omega_{P,2}}}
a_P w_{P_1}(x)1_{\omega_{P,2}}(\theta).
\end{equation*}

The argument continues as in the previous section.
Since the upper tiles  $P_2$ are pairwise disjoint,
the collections $\X$ and $\O$ contain at most one bitile. This observation together with ~\eqref{numaiastaggg} implies that  $$\|\sum_{P\in\S'}a_P w_{P_1}(x)1_{\omega_{P,2}}(\theta)\|_{L^{\infty}(\theta)}\le \sigma$$ and $$\|\epsilon_{\omega_k}\|_{V^r(k)}\lesssim_r \sigma$$
An application of Proposition ~\ref{prop:gh098cfhh} ends the proof. 
\end{proof}

\subsection{Arbitrary collection of trees}
\label{arbitrary}

Let $\S'$ be an arbitrary  collection of bitiles which can be written as a
not necessarily disjoint union of finitely many trees 
$$\S'=\bigcup_{\T\in \F}\T.$$
We next show that $\S'$ can be split into a collection of 2-trees like in Section ~\ref{type2tree} and a collection of 1-trees like in Section ~\ref{type1tree}.

For each $\T\in \F$ let $\T^{(2)}$ be the collection
of all bitiles $P\in \S'$ such that $I_P\subseteq I_{\T}$ and 
$\xi_{\T}\in \omega_{P,2}$.
If $\S^{(2)}$ denotes the union of all trees $\T^{(2)}$, then $\S^{(2)}$ 
qualifies as a collection of trees as in Section \ref{type2tree}.

For each $\T\in \F$ let $\T^{(1)}$ be the collection
of all bitiles $P\in \S'\setminus \S^{(2)}$ such that $I_P\subseteq I_{\T}$ and 
$\xi_{\T}\in \omega_{P,1}$.
If $\S^{(1)}$ be the union of all trees $\T^{(1)}$, then $\S^{(1)}$ 
qualifies as a collection of trees as in
Section \ref{type1tree}. The additional geometric assumption
is satisfied since we have exhausted all 2-trees first.

We will denote by $\F^{(2)}$ and $\F^{(1)}$ respectively the two families of trees that arise from the above procedure. An immediate consequence of the results in the previous two subsections is the following theorem.

\begin{theorem}
\label{mainineqyy123}
Let $(a_P)_{P\in\S'}$ satisfy 
\begin{equation*}
\sup_{P\in\S'}\frac{|a_P|}{|I_P|^{1/2}}\le \sigma.
\end{equation*}
For each  $\beta\ge 1$ and $\gamma>0$ define the  exceptional sets
$$E^{(1)}=\{x:\sum_{\T\in\F}1_{I_{\T}}(x)>\beta\},$$
$$E^{(2)}=\bigcup_{\T\in\F^{(2)}}\{x:\|\sum_{P\in\T\atop{|I_P|< 2^{k}}}a_{P}w_{P_1}(x)\|_{V^r(k)}>\gamma\}.$$
Then  for each $x\notin E^{(1)}\cup E^{(2)}$ and each $r>2$ we have 
$$\|(\sum_{P\in\S'\atop{|I_P|\le 2^{k}}}a_{P}w_{P_1}(x)1_{\omega_{P,2}})_{k\in\Z}\|_{M_2^*}\lesssim_{r} (\sigma+\gamma)\beta^{r/4-1/2}.$$
\end{theorem}

\section{Main argument}
\label{section:last}

In this section we present the proof of Theorem ~\ref{main}. 
For each collection of bitiles $\S'\subseteq \S$ define the following operator.
$$
V_{\S'}f(x)=\left\|\left(\sum_{P\in \S'\atop{|I_P|< 2^k}}\langle f,w_{P_1}\rangle w_{P_1}(x)1_{\omega_{P,2}}(\theta)\right)_{k\in\Z}\right\|_{M_{2}^{*}(\theta)}. 
$$

Note that for each $\S'$ the operator $V_{\S'}$ is sublinear as a function of $f$. Also, for each $f$ and $x$ the mapping  $\S'\to T_{\S'}f(x)$ is sublinear as a function of the bitile set $\S'$. We will prove in the following that
\begin{equation}
\label{en.6}
m\{x:V_{\S}1_{F}(x)\gtrsim \lambda\}\lesssim_p \frac{|F|}{\lambda^p},
\end{equation}
for each $F\subseteq \R_+$ of finite measure, each $\lambda>0$ and each $1<p<\infty.$  Then, by invoking the Marcinkiewicz interpolation theorem and restricted weak type interpolation we get for each $1<p<\infty$ that
$$\|V_{\S}f\|_{p}\lesssim_p\|f\|_p.$$

Fix $F$ and $\lambda$. 
We first prove ~\eqref{en.6} in the case $\lambda\le 1$. 
Define the  first exceptional set
$$E:=\{x:M_{p}1_{F}(x)\ge \lambda\}$$
and note that $|E|\lesssim \frac{|F|}{\lambda^p}.$ Since the range of $p$ is open, it thus suffices to prove that for each $\epsilon>0$
\begin{equation}
\label{en.11}
m\{x\in \R: V_{\S_1}1_F(x)\gtrsim \lambda^{1-\epsilon}\}\lesssim_{\epsilon,p}\frac{|F|}{\lambda^p},
\end{equation}
where
$$
\S_1=\{P\in\S:I_P\cap E^{c}\not=\emptyset\}.
$$
Proposition ~\ref{g11a00uugg55} guarantees that $\size(\S_1)\lesssim \lambda$, where the size is understood here with respect to the  function $1_F$. Define $\Delta:=[-\log_2(\size(\S_1))].$ Use the result of Proposition ~\ref{Besselsineq} to split $\S_1$ as a disjoint union $\S_1=\bigcup_{n\ge \Delta}\P_n,$ where $\size(\P_n)\le 2^{-n}$ and each $\P_n$ consists of a family $\F_{\P_n}$ of trees  satisfying
\begin{equation}
\label{e:intp1}
\sum_{\T\in\F_{\P_n}}|I_\T|\lesssim 2^{2n}|F|.
\end{equation}  

Let $\epsilon>0$ be an arbitrary  positive number.
For each $n\ge \Delta$ define $\sigma:=2^{-n}$, $\beta:=2^{3n}\lambda^{p}$, $\gamma:=2^{-n/2}\lambda^{1/2-\epsilon}$. Define $a_P:=\langle 1_F,w_{P_1}\rangle$ for each $P\in\P_n$ and note that the collection $\P_n$ together with the coefficients $(a_P)_{P\in\P_n}$ satisfy the requirements of Theorem ~\ref{mainineqyy123}. Let $\F_{\P_n}^{(2)}$ be the collection of all the 2-trees $\T^{(2)}$ obtained from  the trees $\T\in\F_{\P_n}$ by the procedure described in the beginning of Section ~\ref{arbitrary}. Define the corresponding exceptional sets
$$E^{(1)}_n=\{x:\sum_{\T\in\F}1_{I_{\T}}(x)>\beta\},$$
$$E^{(2)}_n=\bigcup_{\T\in\F^{(2)}}\{x:\|\sum_{P\in\T\atop{|I_P|< 2^{k}}}a_{P}w_{P_1}(x)\|_{V^r(k)}>\gamma\}.$$
By  ~\eqref{e:intp1} and the fact that $\lambda\le 1$ we get
$$|E^{(1)}_{n}|\lesssim 2^{-n}\lambda^{-p}|F|.$$
By Theorem ~\ref{BMOestimy1} and the fact that $\lambda\le 1$, for each $1<s<\infty$ we get
$$|E^{(2)}_{n}|\lesssim \gamma^{-s}\sigma^{s-2}|F|\lesssim 2^{-n(s/2-2)}\lambda^{-s(1/2-\epsilon)}|F|.$$
Define 
$$E^{*}:=\bigcup_{n\ge \Delta}(E^{(1)}_{n}\cup E^{(2)}_{n}).$$ Note that since $\Delta\gtrsim \log_2(\lambda^{-1})$, we have $|E^{*}|\lesssim \lambda^{-p}|F|,$ an estimate which can be seen by using a sufficiently large $s$.

For each $x\notin E^{*}$, Theorem ~\ref{mainineqyy123} guarantees that
\begin{align*}
&\| \sum_{P\in \S_1\atop{|I_P|< 2^k}}\langle1_{F},w_{P_1}\rangle w_{P_1}(x)1_{\omega_{P,2}}(\theta) \|_{M_{2}^*(\theta)}\\
&\quad \le 
\sum_{n\ge \Delta}\| \sum_{P\in \P_n\atop{|I_P|< 2^k}}\langle 1_{F},w_{P_1}\rangle w_{P_1}(x)1_{\omega_{P,2}}(\theta) \|_{M_{2}^*(\theta)}
\\&\quad\lesssim \sum_{n\ge \Delta}n[2^{(3(r/2-1)-1)n}\lambda^{p(r/2-1)}+2^{(3(r/2-1)-1/2)n}\lambda^{p(r/2-1)+1/2-\epsilon}]\\&\quad\lesssim \lambda^{1-2\epsilon},
\end{align*}
if $r$ is chosen sufficiently close to 2, depending on $p$ and $\epsilon$. 
This ends the proof of ~\eqref{en.11}, and hence the proof of  ~\eqref{en.6} in the case $\lambda\le 1$.

We next focus on proving ~\eqref{en.6} in the case $\lambda>1$. In this remaining part of the discussion the size will be  understood with respect to the function $\lambda^{-1}1_F$. Proposition ~\ref{g11a00uugg55} implies that $\size(\S)\lesssim \lambda^{-1}$. Define $\Delta:=[-\log_2(\size(\S))].$ Split  $\S$ as before, as a disjoint union $\S=\bigcup_{n\ge \Delta}\P_n,$ where $\size(\P_n)\le 2^{-n}$ and each $\P_n$ consists of a family $\F_{\P_n}$ of trees  satisfying 
\begin{equation}
\label{e:intpikh}
\sum_{\T\in\F_{\P_n}}|I_\T|\lesssim 2^{2n}\lambda^{-2}|F|.
\end{equation}

For each $n\ge \Delta$ define $\sigma:=2^{-n}$, $\beta:=2^{(p+1)n}$ and $\gamma:=2^{-n/2}$.
Define also $a_P:=\langle \lambda^{-1}1_F,w_{P_1}\rangle$ for each $P\in\P_n$ and note that the collection $\P_n$ together with  the coefficients $(a_P)_{P\in\P_n}$ satisfy the requirements of Theorem ~\ref{mainineqyy123}. Let $\F_{\P_n}^{(2)}$ the collection of all the 2-trees $\T^{(2)}$ obtained from  the trees $\T\in\F_{\P_n}$ by the procedure described in the beginning of the  Section ~\ref{arbitrary}. Define the corresponding exceptional sets 
$$E^{(1)}_n=\{x:\sum_{\T\in\F}1_{I_{\T}}(x)>\beta\},$$
$$E^{(2)}_n=\bigcup_{\T\in\F^{(2)}}\{x:\|\sum_{P\in\T\atop{|I_P|< 2^{k}}}a_{P}w_{P_1}(x)\|_{V^r(k)}>\gamma\}.$$
By  ~\eqref{e:intpikh} and the fact that $\lambda\ge 1$ we get
$$|E^{(1)}_{n}|\lesssim 2^{-(p-1)n}\lambda^{-2}|F|.$$
By Theorem ~\ref{BMOestimy1} and the fact that $\lambda\ge 1$, for each $1<s<\infty$ we get
$$|E^{(2)}_{n}|\lesssim \gamma^{-s}\sigma^{s-2}\lambda^{-2}|F|\lesssim 2^{-n(s/2-2)}\lambda^{-2}|F|. $$
Define 
$$E^{*}:=\bigcup_{n\ge \Delta}(E^{(1)}_{n}\cup E^{(2)}_{n}).$$ Note that since $\Delta\gtrsim \log_2(\lambda)$, we have $|E^{*}|\lesssim \lambda^{-p}|F|,$ an estimate which can be seen by using a sufficiently large $s$.

For each $x\notin E^{*}$, Theorem ~\ref{mainineqyy123} guarantees that
$$
\left\|\left(\sum_{P\in \S\atop{|I_P|< 2^k}}\langle \lambda^{-1}1_{F},w_{P_1}\rangle w_{P_1}(x) 1_{\omega_{P,2}}(\theta)\right)_{k\in\Z}\right\|_{M_{2}^*(\theta)}$$ $$\le \sum_{n\ge \Delta}\left\|\left(\sum_{P\in \P_n\atop{|I_P|< 2^k}}\langle \lambda^{-1}1_{F},w_{P_1}\rangle w_{P_1}(x)1_{\omega_{P,2}}(\theta)\right)_{k\in\Z}\right\|_{M_{2}^*(\theta)}$$ $$ \lesssim \sum_{n\ge \Delta}n2^{(p+1)(r/2-1)n}(2^{-n}+2^{-n/2})$$ $$\lesssim 1,$$
if $r$ is chosen sufficiently close to 2, depending only on $p$. This ends the proof of ~\eqref{en.6}
in the case $\lambda>1$.


\end{document}